\documentclass[12pt]{amsart}
\usepackage[utf8]{inputenc}
 
\usepackage{amssymb}
\usepackage{bm}
\usepackage{graphicx}
\usepackage[centertags]{amsmath}
\usepackage{amsfonts}
\usepackage{amsthm}
\usepackage{enumerate}
\usepackage{tocvsec2}
\usepackage{xcolor}
\usepackage[margin=1in]{geometry}

\newtheorem{theorem}{Theorem}[section]
\newtheorem*{theorem*}{Theorem}

\newtheorem{corollary}[theorem]{Corollary}
\newtheorem{lemma}[theorem]{Lemma}

\newtheorem{proposition}[theorem]{Proposition}

\theoremstyle{definition}

\newcommand{\rr}{\mathbb{R}}
\newcommand{\nn}{\mathbb{N}}

\newcommand{\ee}{\varepsilon}

\begin{document}

\title{Szlenk and $w^*$-dentability indices of $C(K)$}
\author{R.M. Causey}
\begin{abstract} Given any compact, Hausdorff space $K$ and $1<p<\infty$, we compute the Szlenk and  $w^*$-dentability indices of the spaces $C(K)$ and $L_p(C(K))$.   We show that if $K$ is compact, Hausdorff, scattered, $CB(K)$ is the Cantor-Bendixson index of $K$, and $\xi$ is the minimum ordinal such that $CB(K)\leqslant \omega^\xi$, then $Sz(C(K))=\omega^\xi$ and  $Dz(C(K))=Sz(L_p(C(K)))= \omega^{1+\xi}.$   

\end{abstract}
\maketitle

\section{Definitions}

Two ordinal indices, the Szlenk index \cite{Szlenk} and $w^*$-dentability index, have been used to classify and study Asplund spaces.  These indices are distinct, but happen to coincide for a large class of spaces. Indeed, due to a result of H\'{a}jek and Schlumprecht \cite{HajekSchlumprecht}, the Szlenk and $w^*$-dentability indices coincide for those Banach spaces whose Szlenk index lies in the interval $[\omega^\omega, \omega_1]$.  Here, $\omega$ denotes the first infinite ordinal and $\omega_1$ is the first uncountable ordinal.   Since each index has applications to renorming theory, we seek to better understand the relationship between them.  Given a Banach space $X$, a $w^*$-compact subset $K$ of $X^*$, and $\ee>0$, we let $s_\ee(K)$ denote those $x^*\in K$ such that for every $w^*$-neighborhood $V$ of $x^*$, $\text{diam}(V\cap K)>\ee$.   We let $d_\ee(K)$ denote those $x^*\in K$ such that for every $w^*$-open slice $S$ containing $x^*$, $\text{diam}(S\cap K)>\ee$.   We recall that a $w^*$-open slice in $X^*$ is a subset of $X^*$ of the form $\{y^*\in X^*: \text{Re\ }y^*(x)>a\}$ for some $a\in \rr$ and $x\in X$.    Of course, $s_\ee(K)\subset d_\ee(K)$.  We define $$s^0_\ee(K)=d_\ee^0(K)=K,$$ $$s^{\xi+1}_\ee(K)=s_\ee(s^\xi_\ee(K)),\hspace{5mm} d^{\xi+1}_\ee(K)=d_\ee(d^\xi_\ee(K)),$$ and if $\xi$ is a limit ordinal, $$s^\xi_\ee(K)=\underset{\zeta<\xi}{\bigcap} s_\ee^\zeta(K), \hspace{5mm} d^\xi_\ee(K)=\underset{\zeta<\xi}{\bigcap} d_\ee^\zeta(K).$$ Note that for every $\ee>0$ and every ordinal $\xi$, $s_\ee^\xi(K), d_\ee^\xi(K)$ are $w^*$-compact, and $s_\ee^\xi(K)\subset d_\ee^\xi(K)$.  Moreover, if $K$ is convex, so is $d^\xi_\ee(K)$.   We let $Sz_\ee(K)=\min\{\xi: s^\xi_\ee(K)=\varnothing\}$ if this class is non-empty, and we write $Sz_\ee(K)=\infty$ otherwise.  We let $Sz(K)=\sup_{\ee>0}Sz_\ee(K)$, with the convention that this supremum is $\infty$ if $Sz_\ee(K)=\infty$ for some $\ee>0$.   We define $Dz_\ee(K)$, $Dz(K)$ similarly. Given a Banach space $X$, we let $Sz(X)=Sz(B_{X^*})$ and $Dz(X)=Dz(B_{X^*})$.  If $\Phi:X\to Y$ is an operator, we define $Sz(\Phi)=Sz(\Phi^* B_{Y^*})$.    The index $Sz(X)$ is the \emph{Szlenk index} of $X$, and $Dz(X)$ is the $w^*$-\emph{dentability index} of $X$.  We observe that $Sz(K)\leqslant Dz(K)$ for any $w^*$-compact $K$.  

Given a compact, Hausdorff space $K$, we let $K'$ denote the Cantor-Bendixson derivative of $K$.  That is, $K'$ consists of those points in $K$ which are not isolated in $K$. Note that $K'$ is closed in $K$, and so is compact, Hausdorff with its relative topology as long as it is non-empty.   We define $$K^0=K,$$ $$K^{\xi+1}=(K^\xi)',$$ and if $\xi$ is a limit ordinal, we let $$K^\xi=\underset{\zeta<\xi}{\bigcap} K^\zeta.$$   We let $CB(K)$ denote the Cantor-Bendixson index of $K$, which is the minimum ordinal $\xi$ such that $K^\xi=\varnothing$ if such an ordinal exists, and otherwise we write $CB(K)=\infty$. The space $K$ is said to be \emph{scattered} if $CB(K)$ is an ordinal.  A standard compactness argument yields that if $K$ is scattered, $CB(K)$ is a successor ordinal.  If $K$ is scattered, we let $\Gamma(K)$ denote the minimum ordinal $\xi$ such that $CB(K)\leqslant \omega^\xi$. Note that the inequality $CB(K)\leqslant \omega^\xi$ is strict except in the trivial case that $CB(K)=1$ (that is, when $K$ is finite), since $CB(K)$ cannot be a limit ordinal, while $\omega^\xi$ is a limit ordinal for any $\xi>0$.   If $K$ is not scattered, we let $\Gamma(K)=\infty$.   We agree to the convention that $\omega \infty=\infty$.

Our proofs work for both the real and complex scalars.  In what follows, $C(K)$ is the Banach space of continuous, scalar-valued functions defined on the compact, Hausdorff space $K$.  Given a Banach space $X$ and $1<p<\infty$, $L_p(X)$ denotes the space of (equivalence classes of) Bochner-integrable, $X$-valued functions defined on $[0,1]$ with Lebesgue measure.  

\begin{theorem} For any compact, Hausdorff $K$ and any $1<p<\infty$, $$Sz(C(K))=\Gamma(K),$$ $$Dz(C(K))=Sz(L_p(C(K)))=\omega \Gamma(K).$$   

\label{main theorem}
\end{theorem}

For any ordinals $\xi, \zeta$ such that $\omega^\zeta\leqslant \xi<\omega^{\zeta+1}$, it is easy to see that $CB([0,\xi])=\zeta+1$.  Thus our results recover the known facts that if $\omega^{\omega^\zeta}\leqslant \xi <\omega^{\omega^{\zeta+1}}$, $Sz(C([0, \xi]))=\omega^{\zeta+1}$.   This was shown by Samuel \cite{Samuel} in the case that $\xi$ is countable, by Lancien and H\'{a}jek when $\xi<\omega_1\omega$, and by Brooker \cite{BrookerOrdinals} in the general case.  We also recover the values of $Dz([0, \xi])$ and $Sz(L_p(C([0, \xi])))$, which was shown for countable $\xi$ in \cite{HajekLancienProchazka}, and in the general case by Brooker \cite{BrookerOrdinals}.

\section{Preliminaries}

We collect a few facts concerning the Szlenk and $w^*$-dentability indices.  

\begin{proposition}Let $X$ be a Banach space.  \begin{enumerate}[(i)]\item If $X$ is isomorphic to a subspace of $Y$, then $Sz(X)\leqslant Sz(Y)$ and $Dz(X)\leqslant Dz(Y)$.     \item $Sz(X)=1$ if and only if $\dim X<\infty$, and otherwise $Sz(X)\geqslant \omega$.  \item $Dz(X)=1$ if and only if $X=\{0\}$, and otherwise $Dz(X)\geqslant \omega$. \item $Sz(X)=\infty$ if and only if $Dz(X)=\infty$ if and only if $X$ fails to be Asplund.  \item For any Asplund space $X$, there exist ordinals $\xi, \zeta$ such that $Sz(X)=\omega^\xi$ and $Dz(X)=\omega^\zeta$. \item For any $1<p<\infty$, $Dz(X)\leqslant Sz(L_p(X))$. \item For any Banach space $Y$, $Sz(X\oplus Y)=\max\{Sz(X), Sz(Y)\}$.  \item If $Y$ is a Banach space and $\Phi:X\to Y$ is an isomorphic embedding, $Sz(\Phi)=Sz(X)$.  \item For any $K,L\subset X^*$ $w^*$-compact, $\ee>0$, and any ordinal $\xi$, if $K\subset L+\ee B_{X^*}$, then  $$s_{12\ee}^\xi(K)\subset s_{3\ee}^\xi(L)+\ee B_{X^*}.$$   \end{enumerate}

\label{szlenk facts}

\end{proposition}

Items $(i)$-$(vi)$ can be found in the survey paper \cite{LancienSurvey}. Item $(vii)$ was stated in \cite{HajekLancien} in the case that $Y=X$, but the proof yields the version here.     Item $(viii)$ is due to Brooker \cite{BrookerAsplund}. 

The idea for  $(ix)$ is in \cite[Lemma $6.2$]{Lancien}, alhough the statement was slightly different.  We give the proof of item $(ix)$ as it is stated here.

\begin{proof}[Proof of $(ix)$] By induction.  The $\xi=0$ case is clear.  Assume $\xi$ is a limit ordinal and the result holds for every $\zeta<\xi$, and fix $x^*\in s^\xi_{12\ee}(K)$.  For every $\zeta<\xi$, since $x^*\in s_{12\ee}^\zeta(K)\subset s_{3\ee}^\zeta(L)+\ee B_{X^*}$, we may fix $y^*_\zeta\in s_{3\ee}^\zeta(L)$ and $z^*_\zeta\in \ee B_{X^*}$ such that $x^*=y^*_\zeta+ z^*_\zeta$.  We pass to a subnet $(y^*_\lambda)_{\lambda\in D}$ of $(y^*_\zeta)_{\zeta<\xi}$ with $w^*$-limit $y^*\in s^\xi_{3\ee}(L)$ and note that over the same subnet, $z^*_\lambda \underset{\lambda\in D, w^*}{\to}x^*-y^*\in \ee B_{X^*}$.  Therefore $x^*=y^*+z^*\in s_{3\ee}^\xi(L)+\ee B_{X^*}$.  Last, assume the result holds for some ordinal $\xi$ and suppose $x^*\in s^{\xi+1}_{12\ee}(K)$.   Fix a net $(x^*_\lambda)\subset s^\xi_{12\ee}(K)$ converging $w^*$ to $x^*$ and such that for every $\lambda$, $\|x^*_\lambda- x^*\|>6\ee$.   For every $\lambda$, fix $y^*_\lambda\in s_{3\ee}^\xi(L)$ and $z^*_\lambda\in \ee B_{X^*}$ such that $x^*_\lambda=y^*_\lambda+z^*_\lambda$.  By passing to a subnet, we may assume $y^*_\lambda\underset{w^*}{\to} y^*\in s_{3\ee}^\xi(L)$ and note that over the same subnet, $z^*_\lambda\underset{w^*}{\to} x^*-y^*\in \ee B_{X^*}$.   For every $\lambda$, $$\|y^*_\lambda - y^*\| = \|x^*_\lambda-z^*_\lambda- x^* +x^*-y^*\| \geqslant \|x^*_\lambda - x^*\|-\|z^*_\lambda\|-\|x^*-y^*\| >6\ee-2\ee>3\ee.$$   From this it follows that $y^*\in s_{3\ee}^{\xi+1}(L)$.

\end{proof}

Rudin \cite{Rudin} showed that if $K$ is compact, Hausdorff, scattered, $C(K)^*=\ell_1(K)$, where the canonical $\ell_1(K)$ basis is the set of Dirac functionals $\{\delta_x: x\in K\}$.  By a result of Namioka and Phelps \cite{NamiokaPhelps}, if $K$ is scattered, $C(K)$ is Asplund.  We note that a Banach space $X$ is an Asplund space if and only if every separable subspace of $X$ has separable dual.  This was shown in \cite{DevilleGodefroyZizler} in the real case, and it is explained in \cite{BrookerAsplund} how to deduce the complex case from the real case.  Stegall \cite{Stegall} showed that if every separable subspace of $X$ has separable dual, then $X^*$ has the Radon-Nikodym property.   It follows from \cite[Page 98]{DiestelUhl} that if $X^*$ has the Radon-Nikodym property, for any $1<p<\infty$, $L_p(X)^*=L_q(X^*)$ via the canonical embedding of $L_q(X^*)$ into $L_p(X)^*$. Here, $1/p+1/q=1$.  Therefore if $K$ is scattered and $1<p<\infty$, $L_p(C(K))^*=L_q(\ell_1(K))$.

Given a closed subset $F$ of $K$, we let $C_F(K)$ denote those $f\in C(K)$ such that $f|_F\equiv 0$. If $F=\varnothing$, we let $C_F(K)=C(K)$. If $K$ is scattered, we let $K_\infty=K^{CB(K)-1}$. This is  well-defined, since $CB(K)$ is a successor ordinal. We let $C_0(K)=C_{K_\infty}(K)$.    Note that $K_\infty$ is finite and non-empty, so $0<\dim C(K)/C_0(K)<\infty$.   From this it follows that $L_p(C(K))$ is isomorphic to $L_p(C_0(K))\oplus L_p$. Then if $K$ is infinite,  \begin{align*} Sz(L_p(C(K))) & =\max\{Sz(L_p(C_0(K))), Sz(L_p)\}=\max\{Sz(L_p(C_0(K))), \omega\} \\ & =Sz(L_p(C_0(K))).\end{align*} If $K$ is finite, $K_\infty=K$ and $C_0(K)=\{0\}$, so that $$Sz(L_p(C(K)))=Sz(L_p)=\omega,$$ since $L_p$ is asymptotically uniformly smooth.   This shows that in the non-trivial case that $K$ is infinite, to compute the Szlenk index of $L_p(C(K))$, it is sufficient to compute the Szlenk index of $L_p(C_0(K))$.

With $F\subset K$ still closed, we let $\approx_F $ denote the equivalence relation on $K$ given by $s\approx_F t$ if $s=t$ or if $s,t\in F$.   We let $K/F$ denote the space of equivalence classes $[s]$ of members $s$ of $K$ and endow $K/F$ with the quotient topology coming from the function $\chi_F$ given by $\chi_F(s)\mapsto [s]$. Of course, the equivalence classes in $K/F$ are $F$ and $\{s\}$, $s\in K\setminus F$.   Note that $C_F(K)$ is canonically isometrically isomorphic to $C_{\{F\}}(K/F)$ via the operator $T:C_{\{F\}}(K/F)\to C_F(K)$ given by $Tf(t)=f(\chi_F(t))$.   If $F=K^\gamma$ for some $\gamma<CB(K)$, it is straightforward to check that for $0\leqslant \zeta\leqslant \gamma$, then $(K/K^\gamma)^\zeta=q_F(K^\zeta)$.  In particular, $(K/K^\gamma)^\gamma=\{K^\gamma\}$, and $CB(K/K^\gamma)=\gamma+1$.   This means that $(K/K^\gamma)_\infty$, the last non-empty Cantor-Bendixson derivative of $K/K^\gamma$, is the space consisting of the single equivalence class $K^\gamma$.   From this and our previous remarks, it follows that $C_{K^\gamma}(K)$ is isometrically isomorphic to $C_0(K/K^\gamma)$.

Note that the restriction map $\rho_F:C(K)\to C(F)$ given by $f\mapsto f|_F$ is a quotient map by the Tietze extension theorem. The adjoint $\rho_F^*:\ell_1(F)\to \ell_1(K)$ is the inclusion, and $\rho^*_F B_{\ell_1(F)} =\{\mu\in B_{\ell_1(K)}: |\mu|(K\setminus F)=0\}$, which we identify with $B_{\ell_1(F)}$ throughout.  Note that this identification is a linear, $w^*$-$w^*$-continuous isometry, so that for any ordinal $\xi$ and any $\ee>0$, $s^\xi_\ee(B_{\ell_1(F)})= s^\xi_\ee(\rho^*_F B_{\ell_1(F)})$ and $d^\xi_\ee(B_{\ell_1(F)})= d^\xi_\ee(\rho^*_F B_{\ell_1(F)})$.   We will use this fact throughout.  Moreover, if $R_F:L_p(C(K))\to L_p(C(F))$ is the restriction given by $R_F f(t)= f(t)|_F$, $R_F^*:L_q(\ell_1(F))\to L_q(\ell_1(K))$ is the inclusion, and we may identify $B_{L_q(\ell_1(F))}$ with its image under $R^*_F$ when computing the Szlenk and $w^*$-derivations.    

 Let $\varphi_F:C_F(K)\to C(K)$, $\Phi_F:L_p(C_F(K))\to L_p(C(K))$ denote the inclusions.   Moreover, with this identification, Note that $L_q(\ell_1(F))$ is canonically included in $L_q(\ell_1(K))=L_p(C(K))^*$, and $$L_p(C_F(K))^\perp = L_q(C_F(K)^\perp)=L_q(\ell_1(F)).$$ From this it follows that $L_q(\ell_1(F))$ is $w^*$-closed in $L_q(\ell_1(K))$ and $$L_p(C_F(K))^*= L_q(\ell_1(K))/L_q(\ell_1(F)).$$   Note that any operator $p:\ell_1(K)\to \ell_1(K)$ extends to a function $P:L_q(\ell_1(K))\to L_q(\ell_1(K))$ given by $(Pf)(t)=p(f(t))$, and $\|P\|=\|p\|$.   Let $p_F:\ell_1(K)\to \ell_1(F)\subset \ell_1(K)$ be the canonical projection, and let $q_F$ denote the complementary projection $I_{\ell_1(K)}-p_F$. Note that $\|q_F\|, \|p_F\|\leqslant 1$. Given an ordinal $\gamma<CB(K)$, we let $p_\gamma$, $q_\gamma$ denote the projections $p_{K^\gamma}$, $q_{K^\gamma}$.  Let $P_F, Q_F:L_q(\ell_1(K))\to L_q(\ell_1(F))$ be the maps induced by $p_F$, $q_F$, respectively.  Let $P_\gamma$, $Q_\gamma$ be the maps induced by $p_\gamma$, $q_\gamma$.  Note that the quotient map $\varphi_F^*:\ell_1(K)\to \ell_1(K)/\ell_1(F)$ is given by $$\varphi^*_F(\mu)=q_F(\mu)+\ell_1(F),$$ and $\|\phi^*_F (\mu)\|=\|q_F (\mu)\|$.   The quotient map $\Phi_F^*:L_q(\ell_1(K))\to L_q(\ell_1(K))/L_q(\ell_1(F))$ is given by $$\Phi^*_F(f)= Q_F(f)+L_q(\ell_1(F)),$$   and $\|\Phi^*_F(f)\|=\|Q_F(f)\|$.

Note also that for any $f\in L_q(\ell_1(K))$, $$\|P_Ff\|^q+\|Q_F f\|^q\leqslant \|f\|^q.$$   Indeed, for any $t\in [0,1]$, $$\|p_F f(t)\|^q+\|q_F f(t)\|^q \leqslant (\|p_F f(t)\|+\|q_F f(t)\|)^q = \|f(t)\|^q.$$  Integrating over $t$ yields the inequality.  In particular, if $\|Q_F f\|^q=\|\Phi^*_F f\|^q \geqslant 1-\ee^q$ and $\|f\|\leqslant 1$, then $\|P_F f\|^q \leqslant 1-(1-\ee^q)=\ee^q$.   Thus if $f,g\in B_{L_q(\ell_1(K))}$, $\|f-g\|>3\ee$, and $\|\Phi^*_F f\|^q, \|\Phi^*_F g\|^q>1-\ee^q$, then $$\|\Phi^*_F f- \Phi^*_F g\| = \|Q_F f- Q_Fg\| \geqslant \|f-g\|-\|P_F f\|-\|P_F g\|>\ee.$$   We also note that if $\|\varphi^*_F \mu\|, \|\varphi^*_F \mu'\|>1-\ee$ and $\|\mu-\mu'\|>3\ee$, then $\|\varphi^*_F \mu-\varphi^*_F \mu'\|>\ee$.

We will use these fact to prove the following.  Item $(i)$ is based on \cite[Lemma $3.3$]{HajekLancien} and $(ii)$ is based on \cite[Lemma $6$]{HajekLancienProchazka}.  

\begin{lemma} Suppose $\ee>0$, $F$ is a closed subset of $K$, and $\beta$ is an ordinal. \begin{enumerate}[(i)]\item If $\mu\in s_{3\ee}^\beta(B_{\ell_1(K)})$ and $\|\varphi^*_F \mu\|>1-\ee$, $\varphi^*_F \mu\in s^\beta_\ee(\varphi^*_F B_{\ell_1(K)})$. \item If $f\in s_{3\ee}^\beta(B_{L_q(\ell_1(K))})$ and $\|\Phi^*_F f\|^q>1-\ee^q$, then $\Phi^*_F f\in s_\ee^\beta(\Phi^*_F B_{L_q(\ell_1(K))})$.  \end{enumerate}

\label{lemma1}
\end{lemma}

\begin{proof}$(i)$ We work by induction on $\beta$, with the base and limit ordinal cases clear.  Assume $\mu\in s^{\beta+1}_{3\ee}(B_{\ell_1(K)})$ and $\|\varphi^*_F \mu\|>1-\ee$.   Then there exist two nets $(\mu_\lambda)$, $(\eta_\lambda)$ contained in $s_{3\ee}^\beta(B_{\ell_1(K)})$ indexed by the same directed set, converging $w^*$ to $\mu$, and such that $\|\mu_\lambda-\eta_\lambda\|>3\ee$ for all $\lambda$.  By passing to a subnet and using $w^*$-$w^*$ continuity, we may assume that $\|\varphi^*_F \mu_\lambda\|$, $\|\varphi^*_F \eta_\lambda\|>1-\ee$.  From this it follows that $\varphi^*_F \mu_\lambda, \varphi^*_F \eta_\lambda\in s_\ee^\beta(\varphi^*_F B_{\ell_1(K)})$.  By our previous remarks, $\|\varphi^*_F \mu_\lambda -\varphi^*_F \eta_\lambda\|>\ee$, and since $\varphi^*_F \mu_\lambda$, $\varphi^*_F \eta_\lambda\underset{w^*}{\to} \varphi^*_F \mu$, we deduce that $\varphi^*_F \mu\in s^{\beta+1}_\ee(\varphi^*_F B_{\ell_1(K)})$.

$(ii)$ This follows from an inessential modification of $(i)$.

\end{proof}

\begin{proposition} Let $F\subset K$ be a closed subset of the compact, Hausdorff space $K$.  \begin{enumerate}[(i)]\item If $F$ is a finite set of isolated points, $Sz(\varphi_{K\setminus F})=1$. \item If $F$ is a finite, non-empty set of isolated points, $Sz(\Phi_{K\setminus F})=\omega$.  \item For any $\mu\in \ell_1(K)$ and any real number $a>0$, if $\|\varphi^*_1 \mu\|>a$, then there exists a finite set $F$ of isolated points such that $\|\varphi^*_{K\setminus F} \mu\|>a$.  \item For any $f\in L_q(\ell_1(K))$ and any real number $a>0$, if $\|\Phi^*_1 f\|>a$, then there exists a finite set $F$ of isolated points such that $\|\Phi^*_{K\setminus F} f\|>a$. \item If $\xi<CB(K)$ is a limit ordinal, $\mu\in \ell_1(K)$, and $\|\varphi^*_\xi \mu\|>a$, then there exists an ordinal $\gamma<\xi$ such that $\|\varphi^*_\gamma \mu\|>a$.  \item If $\xi<CB(K)$ is a limit ordinal, $f\in L_q(\ell_1(K))$, and $\|\Phi^*_\xi f\|>a$, then there exists an ordinal $\gamma<\xi$ such that $\|\Phi^*_\gamma f\|>a$.   \end{enumerate}

\label{barry}

\end{proposition}

\begin{proof} We use Proposition \ref{szlenk facts} for $(i)$ and $(ii)$. 

$(i)$ This follows from the fact that the inclusion operator $\varphi_{K\setminus F}:C_{K\setminus F}(K)\to C(K)$ has the same Szlenk index as $C_{K\setminus F}(K)$, which is $1$, since $C_{K\setminus F}(K)$ has finite dimension.

$(ii)$ This follows from the fact that $\Phi_{K\setminus F}$ has the same Szlenk index as $L_p(C_{K\setminus F}(K))\approx L_p$.

$(iii)$ This follows from regularity and the fact that $K\setminus K'=K\setminus K^1$ is the set of isolated points in $K$.   If $\|q_1 \mu\|=|\mu|(K\setminus K')>a$, then there exists a compact, and therefore finite, subset $F$ of $K\setminus K'$ such that $\|q_{K\setminus F} \mu\|=|\mu|(F)>a$.

$(iv)$ One uses $(iii)$ to deduce the result first for simple functions and then extend by density.

$(v)$ This follows from regularity and the fact that any compact subset $F$ of $K\setminus K^\xi$ is contained in $K\setminus K^\gamma$ for some $\gamma<\xi$. Indeed, if such a gamma did not exist, we could choose for every $\gamma<\xi$ some $x_\gamma\in F\cap K^\gamma$.  Any convergent subnet of the net $(x_\gamma)_{\gamma<\xi}$ necessarily converges to a member of $K^\xi\subset K\setminus F$, contradicting the compactness of $F$.   If $\|q_\xi \mu\|=|\mu|(K\setminus K^\xi)>a$, there exists $\gamma<\xi$ and a compact subset $F$ of $K\setminus K^\gamma$ such that $|\mu|(F)>a$.   Then $\|q_\gamma \mu\|=|\mu|(K\setminus K^\gamma)\geqslant |\mu|(F)>a$.

$(vi)$ One uses $(v)$ to deduce the result first for simple functions and then extend by density.

\end{proof}

We conclude this section with a technical fact.  The general idea behind this fact is well-known.  We recall the Hessenberg sum of two ordinals, the details of which can be found in \cite{Monk}.  Given a non-zero ordinal $\xi$, we may write $\xi=\omega^{\gamma_1}n_1+\ldots + \omega^{\gamma_k}n_k$, where $\gamma_1>\ldots >\gamma_k$ and $k, n_i\in \nn$.  Given two ordinals $\xi, \zeta$, by representing them in Cantor normal and then including zero terms if necessary, we may assume there exist $k\in \nn$, $m_i, n_i\in \nn\cup\{0\}$, and $\gamma_1>\ldots >\gamma_k$ such that $\xi=\omega^{\gamma_1}m_1+\ldots + \omega^{\gamma_k}m_k$ and $\zeta=\omega^{\gamma_1}n_1+\ldots + \omega^{\gamma_k}n_k$.  Then the Hessenberg sum of $\xi, \zeta$ is $$\xi\oplus \zeta=\omega^{\gamma_1}(m_1+n_1)+\ldots + \omega^{\gamma_k}(m_k+n_k).$$  We remark that $(\xi\oplus \zeta)+1=(\xi+1)\oplus \zeta$.  We also note that for any ordinal $\xi$, the set $\{(\zeta_1, \zeta_2): \zeta_1\oplus \zeta_2=\xi\}$ is finite. Last, if $\xi$ is any ordinal and $r\in \nn$, $\{(\zeta_1, \zeta_2): \zeta_1\oplus \zeta_2=\omega^\xi r\}=\{(\omega^\xi k, \omega^\xi(r-k): 0\leqslant k\leqslant r\}$.

\begin{lemma} Suppose $B, L, A\subset X^*$ are $w^*$-compact and $B\subset L+A$. \begin{enumerate}[(i)]\item For any $\ee>0$, $s_{2\ee}(B)\subset (s_\ee(L)+A)\cup (L+s_\ee(A))$. \item   For any ordinal $\xi$ and any $\ee>0$,  $$s_{4\ee}^\xi(B)\subset \underset{\zeta\oplus\eta=\xi}{\bigcup} [s_\ee^\zeta(L)+ s_\ee^\eta(A)].$$ \item If for some $\ee>0$, some ordinal $\xi$, and some $r\in \nn$,  $Sz_{3\ee}(L)<\omega^\xi r$, then $$s_{12\ee}^{\omega^\xi r}(B)\subset L+s_{3\ee}^{\omega^\xi}(A).$$   \end{enumerate}

\label{smooth}
\end{lemma}

\begin{proof}$(i)$ If $x^*\in s_{4\ee}(B)$, we may fix a net $(x^*_\lambda)\subset B$ converging $w^*$ to $x^*$ and such that $\|x^*_\lambda- x^*\|>2\ee$ for all $\lambda$. For each $\lambda$, write $x^*_\lambda= y^*_\lambda+z^*_\lambda\in L+A$.  We may pass to a subnet twice and assume $y^*_\lambda\underset{w^*}{\to} y^*\in L$, $z^*_\lambda\underset{w^*}{\to} z^*\in A$, and either $\|y^*_\lambda- y^*\|\geqslant \|z^*_\lambda - z^*\|$ for all $\lambda$ or $\|y^*_\lambda - y^*\| \leqslant \|z^*_\lambda-z^*\|$ for all $\lambda$.   In the first case, it follows that $y^*\in s_\ee(L)$, and in the second, $z^*\in s_\ee(A)$.  Since $y^*+z^*=x^*$, we deduce the first statement.

$(ii)$ We work by induction on $\xi$. The $\xi=0$ case is trivial.   Assume $$s^\xi_{4\ee}(B)\subset  \underset{\zeta\oplus \eta=\xi}{\bigcup} [s_\ee^\zeta(L)+ s_\ee^\eta(A)]$$ and $x^*\in s_\ee^{\xi+1}(B)$.     Fix a net $(x^*_\lambda)\subset s^\xi_{4\ee}(B)$ converging $w^*$ to $x^*$ such that $\|x^*_\lambda - x^*\|>2\ee$ for all $\lambda$.  Since $\{(\zeta, \eta): \zeta\oplus \eta=\xi\}$ is finite, we may pass to a subnet and assume that for some $\sigma, \tau$ with $\sigma\oplus \tau=\xi$, $x^*_\lambda\in s^\sigma_\ee(B)+s^\tau_\ee(L)$ for all $\lambda$.   Then $x^*\in s_{2\ee}(s^\sigma_\ee(L)+s^\tau_\ee(A))$, and we deduce by $(i)$ that $$x^*\in (s^{\sigma+1}_\ee(L)+s^\tau_\ee(A))\cup (s^\sigma_\ee(L)+s^{\tau+1}_\ee(A))\subset \underset{\zeta\oplus \eta=\xi+1}{\bigcup}[ s^\zeta_\ee(L)+ s^\eta_\ee(A)].$$  

Last, assume $\xi$ is a limit ordinal and the result holds for every $\zeta<\xi$.   For every $\zeta<\xi$, $x^*\in s^{\zeta+1}_{4\ee}(B)$, and there exist $\alpha_\zeta$, $\beta_\zeta$ such that $\alpha_\zeta\oplus \beta_\zeta=\zeta+1$ and $x^*\in s^{\alpha_\zeta}_\ee(L)+s^{\beta_\zeta}_\ee(A)$.  By \cite[Proposition $2.5$]{Causey}, there exist a subset $S$ of $[0, \xi)$ and ordinals $\alpha, \beta$ with $\alpha\oplus\beta=\xi$, and such that either $$\alpha\text{\ is a limit ordinal,}\hspace{5mm}\sup_{\zeta\in S}\alpha_\zeta=\alpha, \hspace{5mm} \min_{\zeta\in S}\beta_\zeta\geqslant \beta,$$ or $$\beta\text{\ is a limit ordinal,}\hspace{5mm} \sup_{\zeta\in S} \beta_\zeta=\beta, \hspace{5mm} \min_{\zeta\in S} \alpha_\zeta\geqslant \alpha.$$  In either case, $$x^*\in s^\alpha_\ee(L)+s^\beta_\ee(A)\subset \underset{\zeta\oplus \zeta=\xi}{\bigcup} [s^\zeta_\ee(L)+s^\eta_\ee(A)].$$ Indeed, in the first case, for every $\zeta\in S$, we may fix $y^*_\zeta\in s^{\alpha_\zeta}_\ee(L)$ and $z^*_\zeta\in S^{\beta_\zeta}_\ee(A)\subset s^\beta_\ee(A)$ such that $x^*=y^*_\zeta+z^*_\zeta$.     By passing to a $w^*$-converging subnet $(y^*_\zeta)_{\zeta\in D}$ of $(y^*_\zeta)_{\zeta\in S}$, we note that the $w^*$-limit must lie in $\cap_{\zeta\in S}s^\zeta_\ee(L)=s^\alpha_\ee(L)$, and over the same subnet, $z^*_\zeta\underset{w^*}{\to} x^*-y^*\in s^\beta_\ee(A)$. Thus $x^*=y^*+z^*\in s_\ee^\alpha(L)+s_\ee^\beta(A)$.  The other case is identical.

$(iii)$ For the final statement, note that since $s_{3\ee}^{\omega^\xi r}(L)=\varnothing$,  $$s^{\omega^\xi r}_{12\ee}(B)\subset \underset{k=0}{\overset{r}{\bigcup}}[s^{\omega^\xi k}_{3\ee}(L)+s^{\omega^\xi(r- k)}_{3\ee}(A)]= \underset{k=0}{\overset{r-1}{\bigcup}}[s^{\omega^\xi k}_{3\ee}(L)+s^{\omega^\xi (r-k)}_{3\ee}(A)] \subset L+s^{\omega^\xi}_{3\ee}(A).$$   

\end{proof}

\section{The upper estimates}

Recall that for $\gamma<CB(K)$, $\varphi_\gamma$ denotes the canonical inclusion of $C_{K^\gamma}(K)$ into $C(K)$, $C_{K^\gamma}(K)$ is isomorphic to $C_0(K/K^\gamma)$, and $CB(K/K^\gamma)=\gamma+1$.   Moreover, $\Phi_\gamma$ denotes the canonical inclusion of $L_p(C_{K^\gamma}(K))$ into $L_p(C(K))$, and $L_p(C_{K^\gamma}(K))$ is isomorphic to $L_p(C_0(K/K^\gamma))$.  

\begin{theorem} Given an ordinal $\xi$, consider the following statements: \begin{enumerate}[(i)]\item $A_\xi$: For any compact, Hausdorff space $K$ such that $CB(K)\leqslant \omega^\xi$, $Sz(C_0(K))\leqslant Sz(C(K))\leqslant \omega^\xi$.  \item $B_\xi$: If $\gamma<\omega^\xi$, then for any compact, Hausdorff space $K$ such that $\gamma<CB(K)$, if $\varphi_\gamma:C_{K^\gamma}(K)\to C(K)$ denotes the inclusion, $Sz(\varphi_\gamma)\leqslant \omega^\xi$. \item $C_\xi$: For any compact, Hausdorff $K$ such that $\omega^\xi<CB(K)$, and any $\ee\in (0,1)$, $$\varphi_{\omega^\xi}^*(s^{\omega^\xi}_{3\ee}(B_{\ell_1(K)}))\subset (1-\ee)\varphi^*_{\omega^\xi} B_{\ell_1(K)}.$$ \end{enumerate}

For every ordinal $\xi$, $A_\xi$, $B_\xi$, $C_\xi$ hold.  

\label{main}

\end{theorem}

\begin{theorem}

 Given an ordinal $\xi$, consider the following statements: \begin{enumerate}[(i)]\item $A_\xi':$ For any compact, Hausdorff space $K$ such that $CB(K)\leqslant \omega^\xi$, $Sz(L_p(C_0(K)))\leqslant Sz(L_p(C(K)))\leqslant \omega^{1+\xi}$.  \item $B_\xi':$ If $\gamma<\omega^\xi$, then for any compact, Hausdorff space $K$ such that $\gamma<CB(K)$, if $\Phi_\gamma:L_p(C_{K^\gamma}(K))\to L_p(C(K))$ denotes the inclusion, $Sz(\Phi_\gamma)\leqslant \omega^{1+\xi}$.  \item $C_\xi':$ For any compact, Hausdorff $K$ such that $\omega^\xi<CB(K)$, and any $\ee\in (0,1)$, $$\Phi^*_{\omega^\xi}(s^{\omega^{1+\xi}}_{3\ee}(B_{L_q(\ell_1(K))}))\subset (1-\ee^q)^{1/q}\Phi^*_{\omega^\xi} B_{L_q(\ell_1(K))} .$$  \end{enumerate}

For every ordinal $\xi$, $A_\xi'$, $B_\xi'$, $C_\xi'$ hold.

\label{main thing}
\end{theorem}

The proofs of Theorems \ref{main} and \ref{main thing} are nearly identical.  We will prove Theorem \ref{main thing}.  In order to prove Theorem \ref{main}, one simply runs the same proof replacing Lemma \ref{lemma1}$(ii)$ and Proposition \ref{barry}$(ii)$, $(iv)$, and $(vi)$ with Lemma \ref{lemma1}$(i)$ and Proposition \ref{barry}$(i)$, $(iii)$, and $(v)$, respectively.  This step of the proof is where our methods necessarily diverge from those used to prove upper estimates for $Sz(C(K))$ and $Sz(L_p(C(K)))$ when $K$ is countable. The only part of the proof which requires work will be to deduce $A_{\xi+1}$ from $C_\xi$.   Given $C_\xi$, it is easy to deduce that if $CB(K)= \omega^\xi+1$, then $Sz(C(K))\leqslant \omega^{\xi+1}$, which is a particular case of $A_{\xi+1}$.  It follows from Bessaga and  Pe\l czy\'{n}ski's isomorphic classification of $C(K)$ spaces for countable $K$ that for two countable, compact, Hausdorff, infinite spaces $K,L$, $C(K)$ is isomorphic to $C(L)$ if and only if $\Gamma(K)=\Gamma(L)$. Therefore for countable $K$, in order to deduce $A_{\xi+1}$ from $C_\xi$, it is sufficient to only consider the special case that $CB(K)=\omega^\xi+1$.  However, this isomorphic classification fails for uncountable compact, Hausdorff spaces, and even for uncountable intervals of ordinals.  More specifically, $[0, \omega_1]$ and $[0, \omega_1 2]$ have the same Cantor-Bendixson index, $\omega_1+1$, while the spaces $C([0, \omega_1])$ and $C([0, \omega_1 2])$ are not isomorphic \cite{Semadeni}.  

\begin{proof}[Proof of Theorem \ref{main thing}] We first prove that $A_\xi'\Rightarrow B_\xi'\Rightarrow C_\xi'$.   Assume $A_\xi'$, $K$ is compact, Hausdorff, $\gamma<\omega^\xi$ and $\gamma<CB(K)$.   Then if $\Phi_\gamma:L_p(C_{K^\gamma}(K))\to L_p( C(K))$ is the inclusion, $$Sz(\Phi_\gamma)=Sz(L_p(C_{K^\gamma}(K)))=Sz(L_p(C_0(K/K^\gamma))),$$ and $CB(K/K^\gamma)=\gamma+1\leqslant \omega^\xi$.  By $A_\xi'$, $Sz(\Phi_\gamma)=Sz(L_p(C_0(K/K^\gamma)))\leqslant \omega^{1+\xi}$.  Thus $A_\xi'\Rightarrow B_\xi'$.  

Next, assume $B_\xi'$ holds.  Assume $f\in s^{\omega^{1+\xi}}_{3\ee}(B_{L_q(\ell_1(K))})$ and, to obtain a contradiction, assume $\|\Phi^*_{\omega^\xi} f\|^q>1-\ee^q$.    If $\xi=0$, then by Proposition \ref{barry}$(iv)$, there exists a finite set $F$ of isolated points such that $\|\Phi^*_{K\setminus F} f\|^q>1-\ee^q$.    If $\xi>0$, by Proposition \ref{barry}$(vi)$, there exists $\gamma<\omega^\xi$ such that $\|\Phi^*_\gamma f\|^q>1-\ee^q$.  Then by Lemma \ref{lemma1}, $\Phi^*_{K\setminus F} f\in s_\ee^\omega(\Phi^*_{K\setminus F} B_{L_q(\ell_1(K))})$ if $\xi=0$, and $\Phi^*_\gamma f\in s_\ee^{\omega^{1+\xi}}(\Phi^*_\gamma B_{L_q(\ell_1(K))})$ if $\xi>0$.   In the case that $\xi=0$, we obtain a contradiction to Proposition \ref{barry}$(ii)$, since in this case $s_\ee^\omega(\Phi^*_{K\setminus F} B_{L_q(\ell_1(K))})=\varnothing$.  In the case that $\xi>0$, we obtain a contradiction to $B_\xi'$, since $Sz(\Phi_\gamma)\leqslant \omega^{1+\xi}$, so $s_\ee^{\omega^{1+\xi}}(\Phi^*_\gamma B_{L_q(\ell_1(K))})=\varnothing$. In either case, the contradiction yields that $\|Q_{K^{\omega^\xi}} f\|=\|\Phi^*_{\omega^\xi} f\|\leqslant (1-\ee^q)^{1/q}$.  Then $$\Phi^*_{\omega^\xi} f= \Phi^*_{\omega^\xi}Q_{K^{\omega^\xi}}f\in (1-\ee^q)^{1/q} \Phi^*_{\omega^\xi} B_{L_q(\ell_1(K))}.$$ Thus $A_\xi'\Rightarrow B_\xi'\Rightarrow C_\xi'$.

We remark that $C'_\xi$ is equivalent to: For any compact, Hausdorff $K$ such that $\omega^\xi<CB(K)$, any $\ee\in (0,1)$, and any $a\in (0,1]$,  $$\Phi^*_{\omega^\xi}(s^{\omega^{1+\xi}}_{3\ee}(a B_{L_q(\ell_1(K))}))\subset a(1-\ee^q)^{1/q}\Phi^*_{\omega^\xi} B_{L_q(\ell_1(K))} .$$ Indeed, by homogeneity, if $X$ is any Banach space, $L\subset X^*$ is $w^*$-compact,  $\ee>0$, $a\in (0,1]$, and $\xi$ is any ordinal, $s^\xi_\ee(a L)=a s^\xi_{\ee/a}(L)\subset a s^\xi_\ee(L)$.  We apply this with $L=B_{L_q(\ell_1(K))}$ to deduce that $$\Phi^*_{\omega^\xi}(s^{\omega^{1+\xi}}_{3\ee}(a B_{L_q(\ell_1(K))}))= a\Phi^*_{\omega^\xi}(s^{\omega^{1+\xi}}_{3\ee/a}( B_{L_q(\ell_1(K))}))\subset a(1-\ee^q)^{1/q}\Phi^*_{\omega^\xi} B_{L_q(\ell_1(K))} .$$

We turn now to the proof of the theorem.  Since $A_\xi'\Rightarrow B_\xi'\Rightarrow C_\xi'$, it is sufficient to prove that for any ordinal $\xi$, $A_\xi'$ holds given $A_\zeta', B'_\zeta$, $C'_\zeta$ for every $\zeta<\xi$.  Seeking a contradiction, assume there exists an ordinal $\xi$ such that $A_\xi'$ fails, and assume that $\xi$ is the minimum such ordinal.  We note that $\xi>0$, since $A_0'$ is true.  Indeed, if $CB(K)\leqslant \omega^0=1$, $K$ is finite, whence $L_p(C(K))\approx L_p$ and $Sz(L_p)=\omega$.   

Assume $\xi$ is a limit ordinal.  Fix a compact, Hausdorff space $K$ with $CB(K)\leqslant \omega^\xi$.   Since $CB(K)$ cannot be a limit ordinal, $CB(K)<\omega^\xi$.  Since $\xi$ is a limit ordinal, there exists $\zeta<\xi$ such that $CB(K)<\omega^\zeta$.  Since $A_\zeta$ holds, we deduce that $Sz(L_p(C(K)))\leqslant \omega^{1+\zeta}< \omega^{1+\xi}$, a contradiction.   Thus $\xi$ cannot be a limit ordinal.   

It follows that $\xi$ must be a successor, say $\xi=\zeta+1$.   Since $\omega^\xi$ is a limit ordinal, $CB(K)=\omega^\xi$ is impossible for any compact, Hausdorff $K$, so it follows that there exists some compact, Hausdorff $K$ such that $CB(K)<\omega^\xi$ and such that the statement of $A_\xi'$ fails for this $K$.  Let $n$ be the minimum natural number such that there exists a compact, Hausdorff space $K$ with $CB(K)\leqslant \omega^\zeta n+1$ and such that the $A_\xi'$ fails for this $K$. First suppose that $n=1$.  Then $CB(K)\leqslant \omega^\zeta+1$.  In this case it must be that $CB(K)=\omega^\zeta+1$, since if $CB(K)\leqslant \omega^\zeta$, we would deduce that $Sz(L_p(C(K)))\leqslant \omega^{1+\zeta}<\omega^{1+\xi}$ by $A_\zeta'$.   Since $K$ is infinite, it must be that $Sz(L_p(C_0(K)))=Sz(L_p(C(K)))$, so we compute $Sz(L_p(C_0(K)))$.   Note that since $K_\infty=K^{\omega^\zeta}$, $\Phi_{\omega^\zeta}$ is the inculsion of $L_p(C_0(K))$ into $L_p(C(K))$.   By $C_\zeta'$, for any $\ee\in (0,1)$, $$\Phi^*_{\omega^\zeta} s_{3\ee}^{\omega^{1+\zeta}}(B_{L_q(\ell_1(K))})\subset (1-\ee^q)^{1/q}\Phi^*_{\omega^\zeta} B_{L_q(\ell_1(K))}.$$ By \cite[Lemma $2.5$]{BrookerAsplund}, $$s_{6\ee}^{\omega^{1+\zeta}}(\Phi^*_{\omega^\zeta} B_{L_q(\ell_1(K))}) \subset \Phi^*_{\omega^\zeta} s_{3\ee}^{\omega^{1+\zeta}}(B_{L_q(\ell_1(K))}).$$    These inclusions together with a standard homogeneity argument yield that for any $j\in \nn$,  and $\ee\in (0,1)$, $$s_{6\ee}^{\omega^{1+\zeta} j }(\Phi^*_{\omega^\zeta} B_{L_q(\ell_1(K))})\subset (1-\ee^q)^{j/q} \Phi^*_{\omega^\zeta} B_{L_q(\ell_1(K))},$$ from which it follows that $Sz_{6\ee}(\Phi^*_{\omega^\zeta} B_{L_q(\ell_1(K))})<\omega^{1+\zeta}\omega=\omega^{1+\xi}$.  This shows that $Sz(\Phi_{\omega^\zeta})\leqslant \omega^{1+\xi}$. Since $Sz(\Phi_{\omega^\zeta})=Sz(L_p(C_0(K)))=Sz(L_p(C(K)))$, we reach a contradiction if $n=1$. 

Since it must be that $n>1$, we may write $n=m+1$. Let $F=K^{\omega^\zeta}$  and note that $CB(F)\leqslant \omega^\zeta m+1$. Indeed, $$F^{\omega^\zeta m+1} = (K^{\omega^\zeta})^{\omega^\zeta m+1}= K^{\omega^\zeta+\omega^\zeta m+1}= K^{\omega^\zeta n+1}=\varnothing.$$     Let $L=B_{L_p(C(F))^*}=B_{L_q(\ell_1(F))}\subset B_{L_q(\ell_1(K))}$. Recall that viewing $B_{L_q(\ell_1(F))}$ as the unit ball of $L_p(C(F))^*$ as well as the subset of $B_{L_q(\ell_1(K))}$ consisting of measures supported on $F$ is a $w^*$-$w^*$-continuous linear isometry, so that $L\subset B_{L_q(\ell_1(K))}$ is $w^*$-compact and $Sz_{3\ee}(L)=Sz_{3\ee}(B_{L_q(\ell_1(F))})=Sz_{3\ee}(B_{L_p(C(F))^*})\leqslant \omega^{1+\xi}$ by the minimality of $n$ and the fact that $CB(F)\leqslant \omega^\zeta m+1$. But the usual compactness argument yields that if $Sz_{3\ee}(L)\leqslant \omega^{1+\xi}$, then $Sz_{3\ee}(L)<\omega^{1+\xi}=\omega^{1+\zeta}\omega$, and there exists some $r\in \nn$ such that $Sz_{3\ee}(L)<\omega^{1+\zeta} r$. We claim that for $i=0, 1, \ldots$ and $f\in s_{12\ee}^{\omega^{1+\zeta}ri} (B_{L_q(\ell_1(K))})$, $\|\Phi^*_F f\|\leqslant (1-\ee^q)^{i/q}$.   The $i=0$ case is obvious.  Assume we have the result for some $i$.  Let $A=(1-\ee^q)^{i/q}B_{L_q(\ell_1(K))}$.  Note that our assumption on $i$ yields that $$s^{\omega^{1+\zeta}ri}_{12\ee}(B_{L_q(\ell_1(K))})=\{P_F f+Q_F f: f\in s^{\omega^{1+\zeta}ri}_{12\ee}(B_{L_q(\ell_1(K))})\}  \subset L+A.$$  Here we are using the fact that $\|P_F\|\leqslant 1$, so that $P_F$ maps $B_{L_q(\ell_1(K))}$ into $L$.      By Lemma \ref{smooth}, $$s^{\omega^{1+\zeta} r(i+1)}_{12\ee}(B_{L_q(\ell_1(K))}) \subset s^{\omega^{1+\zeta}}_{12\ee}\Bigl(s^{\omega^{1+\zeta}ri}_{12\ee}(B_{L_q(\ell_1(K))})\Bigr)\subset s^{\omega^{1+\zeta}}_{12\ee}(L+A)\subset L+s^{\omega^{1+\zeta}}_{3\ee}(A).$$   Using the equivalent version of $C_\zeta'$ given above with $a=(1-\ee^q)^{i/q}$, \begin{align*} \Phi^*_{\omega^\zeta} s^{\omega^{1+\zeta}}_{3\ee}(A) & = \Phi^*_{\omega^\zeta} s^{\omega^{1+\zeta}}_{3\ee}(a B_{L_q(\ell_1(K))})\subset a(1-\ee^q)^{1/q}\Phi^*_{\omega^\zeta} B_{L_q(\ell_1(K))} \\ & =(1-\ee^q)^{\frac{i+1}{q}}\Phi^*_{\omega^\zeta} B_{L_q(\ell_1(K))}.\end{align*} Fix $f\in s^{\omega^{1+\zeta} r(i+1)}_{12\ee}(B_{L_q(\ell_1(K))})\subset L+s^{\omega^{1+\zeta}}_{3\ee}(A)$ and write $f=g+h$ with $g\in L$ and $h\in s^{\omega^{1+\zeta}}_{3\ee}(A)$.  Then since $\Phi^*_F|L\equiv 0$, $\Phi^*_F g=0$.  This fact together with our inclusion above yields that $\|\Phi^*_F f\|= \|\Phi^*_F h\| \leqslant (1-\ee^q)^{\frac{i+1}{q}}.$    This yields the inductive claim on $i$.

 Fix $j\in \nn$ such that $(1-\ee^q)^{j/q}<\ee$.  For any $f\in s^{\omega^{1+\zeta} rj}_{12\ee}(B_{L_q(\ell_1(K))})$, $f=P_{\omega^\zeta} f+Q_{\omega^\zeta} f$, $P_{\omega^\zeta} f\in L$, and $$\|Q_{\omega^\zeta} f\|= \|\Phi^*_{\omega^\zeta} f\|\leqslant (1-\ee^q)^{j/q}<\ee.$$  Therefore $$s^{\omega^{1+\zeta} rj}_{12\ee}(B_{L_q(\ell_1(K))})  \subset L+ \ee B_{L_q(\ell_1(K))}.$$ Using Proposition \ref{szlenk facts}$(vi)$,  \begin{align*} s^{\omega^{1+\zeta}r(j+1)}_{12\ee}(B_{L_q(\ell_1(K))}) & \subset s_{12\ee}^{\omega^{1+\zeta}r}\bigl(s^{\omega^{1+\zeta} rj}_{12\ee}(B_{L_q(\ell_1(K))})\bigr)\subset s_{12\ee}^{\omega^{1+\zeta}r}(L+\ee B_{L_q(\ell_1(K))}) \\ & \subset s_{3\ee}^{\omega^{1+\zeta}r}(L)+\ee B_{L_q(\ell_1(K))}=\varnothing.\end{align*}  But this shows that $Sz_{12\ee}(B_{L_q(\ell_1(K))})<\omega^{1+\zeta}r(j+1)<\omega^{1+\xi}$.  Since $\ee\in (0,1)$ was arbitrary, we deduce that $Sz(B_{L_q(\ell_1(K))})\leqslant \omega^{1+\xi}$, and this contradiction finishes the proof.

\end{proof}

\section{The lower estimates}

\begin{lemma} Let $K$ be compact, Hausdorff and fix $\ee\in (0,1)$.  For any ordinal $\xi<CB(K)$, $B_{\ell_1(K^\xi)}\subset s_\ee^\xi(B_{\ell_1(K)})\cap d_\ee^{\omega\xi}(B_{\ell_1(K)})$.  

\label{lower est}
\end{lemma}

\begin{proof}We prove the results separately for the Szlenk and $w^*$-dentability indices, each by induction.  We first prove that $B_{\ell_1(K^\xi)}\subset s_\ee^\xi(B_{\ell_1(K)})$.  The base and limit ordinal cases are clear.  Assume $K^{\xi+1}\neq \varnothing$ and $B_{\ell_1(K^\xi)}\subset s_\ee^\xi(B_{\ell_1(K)})$.   Fix any finite, non-empty subset $F$ of $K^{\xi+1}$ and scalars $(a_x)_{x\in F}$ such that $\sum_{x\in F}|a_x|\leqslant 1$.  Let $\mu=\sum_{x\in F}a_x\delta_x$ and let $V$ be any $w^*$-neighborhood of $\mu$.   Fix $f_1, \ldots, f_n\in C(K)$, and $\sigma>0$ such that $$\{\eta\in \ell_1(K): (\forall 1\leqslant i\leqslant n)(|\langle \mu, f_i\rangle- \langle \eta, f_i\rangle|<\sigma)\}\subset V.$$     For each $x\in F$, fix a neighborhood $U_x$ of $x$ such that for each $1\leqslant i\leqslant n$ and each $y\in U_x$, $|f_i(x)-f_i(y)|<\sigma$.   We may assume that the sets $(U_x)_{x\in F}$ are pairwise disjoint. Since the points $x\in F$ are not isolated in $K^\xi$, for each $x\in F$ we may fix $y_x\in U_x\cap K^\xi$ such that $y_x\neq x$. Let $E=\{y_x: x\in F\}$. Let $\eta=\sum_{x\in F} a_x \delta_{y_x}$ and note that $\eta\in V\cap B_{\ell_1(K^\xi)}\subset V\cap s_\ee^\xi(B_{\ell_1(K)})$. Since $K^{\xi+1}\neq \varnothing$, $K^\xi$ is infinite, and we may find some $\eta'\in \text{span}\{\delta_t: t\in K^\xi\setminus (E\cup F)\}\cap_{i=1}^n \ker(f_i)$.  By scaling, we may assume $\|\eta'\|=1-\|\eta\|$.   Then $\eta+\eta'\in V\cap B_{\ell_1(K^\xi)}$ and $$\|\mu-(\eta+\eta')\|=\|\mu\|+\|\eta\|+\|\eta'\|\geqslant 1>\ee.$$  Since $V$ was arbitrary, we deduce that $\mu\in s_\ee^{\xi+1}(B_{\ell_1(K)})$.  This shows that the members of $B_{\ell_1(K^\xi)}$ with finite support lie in $s_\ee^{\xi+1}(B_{\ell_1(K)})$.  Since such measures are dense in $B_{\ell_1(K^\xi)}$ and since $s_\ee^{\xi+1}(B_{\ell_1(K)})$ is $w^*$-closed, we deduce the successor case.

We next prove the statement concerning the $w^*$-dentability index, also by induction.  Again, the base and limit cases are trivia. Assume $K^{\xi+1}\neq \varnothing$ and assume $B_{\ell_1(K^\xi)}\subset d^{\omega \xi}_\ee(B_{\ell_1(K)})$.    For each $n=0, 1, \ldots$, let $$A_n=\Bigl\{\frac{1}{2^n}\sum_{t\in F}\ee_t\delta_t: |\ee_t|=1, F\subset K^\xi, |F|=2^n\Bigr\}\subset B_{\ell_1(K^\xi)}\subset d_\ee^{\omega\xi}(B_{\ell_1(K)}).$$   We claim that for every $m=0, 1, \ldots$ and every $n\geqslant m$, $A_n\subset d_\ee^{\omega\xi+m}(B_{\ell_1(K)})$.   The inductive hypothesis yields the result for each $n$ when $m=0$.  Assume that for some $m$ and every $n\geqslant m$, $A_n\subset d_\ee^{\omega\xi+m}(B_{\ell_1(K)})$.    Fix some $n\geqslant m+1$ and some $\mu=\frac{1}{2^n}\sum_{t\in F}\ee_t\delta_t\in A_n$.   Let $F_1, F_2$ be a partition of $F$ with $|F_1|=|F_2|=2^{n-1}$.  Then $\mu_1:=\frac{1}{2^{n-1}}\sum_{t\in F_1}\ee_t\delta_t$, $\mu_2:=\frac{1}{2^{n-1}}\sum_{t\in F_2}\ee_t\delta_t\in A_{n-1}\subset d_\ee^{\omega\xi+m}(B_{\ell_1(K)})$.    Moreover, $\mu=\frac{1}{2}\mu_1+\frac{1}{2}\mu_2$.  Let $S$ be any $w^*$-open slice containing $\mu$, and note that this slice must contain either $\mu_1$ or $\mu_2$ by convexity.  But $$\|\mu-\mu_1\|=\|\mu-\mu_2\|=\frac{1}{2}\|\mu_1-\mu_2\|=1>\ee,$$ so that $\text{diam}(S\cap d_\ee^{\omega \xi+m}(B_{\ell_1(K)}))>\ee$.  From this it follows that $\mu\in d^{\omega\xi +m+1}_\ee(B_{\ell_1(K)})$.   This yields the claim concerning $A_n$.

Next, we note that for any $x\in K^{\xi+1}$, any unimodular scalar $a$, and any $n\in \nn$, $a\delta_x\in \overline{A}_n^{w^*}$.  Indeed, fix $f_1, \ldots, f_k\in C(K)$ and $\eta>0$.  Fix a neighborhood $U$ of $x$ such that for every $y\in U$ and $1\leqslant i\leqslant k$, $|f_i(y)-f_i(x)|<\eta$.  Fix any finite subset $F\subset K^\xi$ of $U$ with $|F|=2^n$, as we may, since $x$ is not isolated in $K^\xi$.  Then if $\mu=\frac{1}{2^n}\sum_{t\in F}a\delta_t\in A_n$, $|\langle f_i, a\delta_x\rangle-\langle f_i, \mu\rangle|<\eta$ for $i=1, \ldots, k$.   This yields that $a\delta_x\in \overline{A}^{w^*}_n\subset \overline{d^{\omega\xi+n}_\ee(B_{\ell_1(K)})}^{w^*}= d^{\omega\xi+n}_\ee(B_{\ell_1(K)})$.   From this we deduce that $a\delta_x\in \cap_{n<\omega}\delta_\ee^{\omega\xi+n}(B_{\ell_1(K)})=d^{\omega\xi+\omega}_\ee(B_{\ell_1(K)})=d^{\omega(\xi+1)}_\ee(B_{\ell_1(K)})$. Since $d^{\omega(\xi+1)}_\ee(B_{\ell_1(K)})$ is $w^*$-closed and convex, it follows that $$B_{\ell_1(K^{\xi+1})} = \overline{\text{co}}^{w^*}\{a \delta_x:|a|=1, x\in K^{\xi+1}\}\subset d_\ee^{\omega(\xi+1)}(B_{\ell_1(K)}).$$

\end{proof}

\begin{corollary} For any compact, Hausdorff space $K$ and any $1<p<\infty$,  $Sz(C(K))=\Gamma(K)$ and $Dz(C(K))=Sz(L_p(C(K)))=\omega \Gamma(K)$. 

\end{corollary}

\begin{proof} We will use Proposition \ref{szlenk facts} throughout.  Since $Dz(C(K))\leqslant L_p(C(K))$, it suffices to prove that $\omega\Gamma(K) \leqslant Dz(C(K))$ and $Sz(L_p(C(K)))\leqslant \omega\Gamma(K)$ in order to see that $Dz(C(K))=Sz(L_p(C(K)))=\omega\Gamma(K)$.

 By Lemma \ref{lower est}, if $K$ is not scattered, $C(K)$, and therefore $L_p(C(K))$, is not Asplund.  From this it follows that $$Sz(C(K))=Dz(C(K))=Sz(L_p(C(K)))= \omega\Gamma(K)=\infty.$$

Assume $K$ is scattered.  In this case, $CB(K)$ is an ordinal and there exists an ordinal $\xi$ such that $\Gamma(K)=\omega^\xi$.  By the definition of $\Gamma(K)$, $CB(K)\leqslant \omega^\xi$.  From Theorems \ref{main} and \ref{main thing}, $Sz(C(K))\leqslant \omega^\xi= \Gamma(K)$ and $Sz(L_p(C(K)))\leqslant \omega^{1+\xi}=\omega\Gamma(K)$.

If $CB(K)=1$, $Sz(C(K))=1$ and $Dz(C(K))=\omega$, since $C(K)$ is finite-dimensional and non-zero. This yields the result in the trivial case that $K$ is finite.   Otherwise $CB(K)=\zeta+1$ for some ordinal $\zeta$ and there exists an ordinal $\xi$ such that $\omega^\xi<\zeta+1<\omega^{\xi+1}=\Gamma(K)$.  By Lemma \ref{lower est},  $B_{\ell_1(K^\zeta)}\subset s_{1/2}^\zeta(B_{\ell_1(K)})\cap d_{1/2}^{\omega\zeta}(B_{\ell_1(K)})$, and it follows that $Sz_{1/2}(B_{\ell_1(K)})>\zeta>\omega^\xi$ and $Dz_{1/2}(B_{\ell_1(K)})>\omega\zeta >\omega^{1+\xi}$. Therefore we deduce that $\omega^\xi<Sz(C(K))\leqslant \omega^{\xi+1}$, and since $Sz(C(K))=\omega^\gamma$ for some ordinal $\gamma$, $Sz(C(K))=\omega^{\xi+1}$. We deduce from $\omega^{1+\xi}<Dz(C(K))\leqslant \omega^{1+\xi+1}$ that $Dz(C(K))\geqslant \omega^{1+\xi+1}$ similarly.

\end{proof}


\begin{thebibliography}{HD}

\normalsize
\baselineskip=17pt

 
\bibitem{BessagaPelczynski}  C. Bessaga, A. Pe\l czy\'{n}ski, \emph{Spaces of continuous functions IV}, Studia Math. 19 (1960) 53-62.
 
\bibitem{BrookerAsplund} P.A.H. Brooker, \emph{Asplund operators and the Szlenk index}, Operator Theory 68 (2012) 405-442.
 
\bibitem{BrookerOrdinals} P.A.H. Brooker, \emph{Szlenk and $w^*$-dentability indices of the Banach spaces $C([0, \alpha])$}, Journal of Mathematical Analysis and Applications 399 (2013) 559-564.

\bibitem{Causey} R.M. Causey, \emph{Proximity to $\ell_p$ and $c_0$ in Banach spaces}, J. Funct. Anal., (2015) 10.1016/j.jfa.2015.10.001. 
 
\bibitem{DiestelUhl} J. Diestel, J. J. Uhl, Jr., \emph{Vector measures}, Mathematical Surveys and Monographs Volume $15$, American Mathematical Society, Rhode Island 1977.  

\bibitem{DevilleGodefroyZizler}  R. Deville, G. Godefroy, V. Zizler, \emph{Smoothness and renormings in Banach spaces}, Pitman Monographs and Surveys in Pure and Applied Mathematics, vol. 64, Longman Scientific \& Technical, Harlow 1993. 

 
\bibitem{HajekLancien} P. H\'{a}jek, G. Lancien, \emph{Various slicing indices on Banach spaces}, Mediterr. J. Math. 4(2007) 179-190.
 
\bibitem{HajekLancienProchazka}  P. H\'{a}jek, G. Lancien, A. Proch\'{a}zka, \emph{Weak$^*$ dentability index of spaces $C([0, \alpha])$}, J. Math. Anal. Appl. 353 (2009) no. 1, 239-243.
 
\bibitem{HajekSchlumprecht} P. H\'{a}jek and Th. Schlumprecht, \emph{The Szlenk index of $L_p(X)$}, Bull. Lond. Math. Soc. 46 (2014)  no. 2, 415-424.

 
\bibitem{Lancien} G. Lancien, \emph{On the Szlenk index and the weak$^*$-dentability index}, Quarterly J. Math. Oxford, 47(1996) 59-71.

\bibitem{LancienSurvey} G. Lancien, \emph{A survey on the Szlenk index and some of its applications}, RACSAM Rev. R. Acad. Cienc. Exactas F'is. Nat. Ser. A Mat. 100(2006), 209-235.

\bibitem{Monk} J.D. Monk. \emph{Introduction to set theory}, McGraw-Hill, (1969). 
 
\bibitem{NamiokaPhelps} I. Namioka, R. R. Phelps, \emph{Banach spaces which are Asplund spaces}, Duke Math. J., 42, (1975) no. 4, 735-750.

\bibitem{Rudin} W. Rudin, \emph{Continuous Functions on Compact Spaces Without Perfect Subsets}, Proc. Amer. Math. Soc., 8 (1957) no. 1, 39-42.

\bibitem{Samuel} C. Samuel, \emph{Indice de Szlenk des C(K), S\'{e}minaire de G\'{e}om\'{e}trie des espaces de Banach}, Vol. I-II, Publications Math\'{e}matiques de l'Universit\'{e} Paris VII, Paris (1983) 81-91.

\bibitem{Semadeni} Z. Semadeni, \emph{Banach spaces non isomorphic to their cartesian squares II}, Bulletin de l'Acad\'{e}mie Polonaise des Sciences, S\'{e}rie Math. Astron. et Phys., 8 (1960) 81-84.

\bibitem{Stegall} C. Stegall, \emph{The Radon-Nikodym Property in Conjugate Banach Spaces}, Trans. of the Amer. Math. Soc., 206 (1975) 213-223.

\bibitem{Szlenk} W. Szlenk, \emph{The non existence of a separable reflexive Banach space universal for all separable reflexive Banach spaces}, Studia Math. 30 (1968) 53-61. 


\end{thebibliography}
\end{document}